\documentclass[11pt,letterpaper,reqno]{amsart}
\usepackage[left=30mm,right=30mm,top=28mm,bottom=28mm]{geometry}

\usepackage{amsmath,amssymb,amsthm}
\usepackage{mathrsfs}

\title[Dirichlet's principle and Perron's method]{A direct proof of the equivalence between Dirichlet's principle and Perron's method}

\author{Tsogtgerel Gantumur}
\address{McGill University\\National University of Mongolia\\Institute of Mathematics and Digital Technology, MAS}
\email{gantumur.tsogtgerel@mcgill.ca}

\date{\today}

\theoremstyle{plain}
\newtheorem{theorem}{Theorem}[section]
\newtheorem{lemma}[theorem]{Lemma}

\theoremstyle{definition}
\newtheorem{remark}[theorem]{Remark}

\numberwithin{equation}{section}

\begin{document}

\begin{abstract}
We give a short proof that for a bounded domain $\Omega\subset\mathbb{R}^n$
and continuous boundary data $g\in C(\partial\Omega)$ admitting a continuous finite-energy extension $\phi\in H^{1}(\Omega)\cap C(\bar\Omega)$,
the minimizer of the Dirichlet energy
\[
  E(v) = \int_{\Omega} |\nabla v|^{2}\,dx,
  \qquad v-\phi\in H^{1}_{0}(\Omega),
\]
coincides with the Perron solution $h_g$
of the Dirichlet problem $\Delta u = 0$ in $\Omega$ with boundary data $g$.
The argument stays entirely in $H^{1}(\Omega)$ and uses only strong convergence via strict convexity
of the Dirichlet energy, Friedrichs' inequality, Weyl's lemma, and Wiener's
exhaustion by regular subdomains.  No weak convergence, Poisson problems with
distributional right hand sides, or general elliptic theory are needed.
\end{abstract}

\maketitle

\section{Introduction}

Let $\Omega\subset\mathbb{R}^{n}$ be a bounded domain and $g$ a function on
$\partial\Omega$.  The classical Dirichlet problem seeks a
harmonic function $u$ in $\Omega$ which takes the boundary values $g$.
Two standard constructions for the solution $u$ are:
\begin{itemize}
\item the \emph{Dirichlet principle}: minimize the Dirichlet energy
\[
E(v)=\int_{\Omega}|\nabla v|^{2}\,dx
\]
over an affine translate of $H^{1}_{0}(\Omega)$ determined by the boundary
values; and
\item \emph{Perron's method}: define $u$ as the upper envelope of subharmonic
functions dominated by $g$ along $\partial\Omega$.
\end{itemize}
Here $H^{1}(\Omega)$ denotes the usual Sobolev space, and $H^{1}_{0}(\Omega)$
is the closure of $C_{c}^{\infty}(\Omega)$, the space of smooth functions with
compact support in $\Omega$, in $H^{1}(\Omega)$.  
Very roughly, Perron's method proceeds by taking all subharmonic functions in $\Omega$
which do not exceed the boundary data $g$ in the sense of limiting boundary values and
then defining $h_g$ as the pointwise supremum of this class; when $h_g$ is finite and
harmonic, it is called the Perron solution associated with $g$.

The Dirichlet principle has its origins in the interpretation of electrostatic equilibrium as the configuration of least field energy, among all potentials with prescribed boundary values.
After initial popularization by Riemann and the later criticism of Weierstrass,
Hilbert put this picture on a rigorous footing via the direct method of calculus of variations, 
and the subsequent work of Sobolev and Weyl made the Hilbert–space setting of the Dirichlet principle completely transparent \cite{Hilbert1904,Sobolev1938,WeylProjection}.
On the other hand, a very flexible method based on subharmonic functions was developed by Poincar\'e under the name {\em balayage},
which inspired Lebesgue's notion of barriers, Perron’s envelope construction, and Wiener’s exhaustion by regular subdomains \cite{Poincare1890,Lebesgue1912,Perron1923,WienerDirichlet}.
These methods form the classical core on which modern potential theory is built, and as such their interrelationships are well understood.
In particular, for continuous boundary data $g$, balayage, Perron's method, and Wiener's exhaustion are known to produce the same harmonic function in $\Omega$.  It is therefore natural to ask whether the variational solution given by the Dirichlet principle also coincides with this classical potential–theoretic solution under comparable hypotheses; this is the question addressed in the present note.

From the point of view of modern potential theory, such an equivalence can already be obtained as a corollary of very general theories.  
This can be deduced, for instance, by invoking the full potential-theoretic machinery of Littman--Stampacchia--Weinberger together with Frostman’s 
Wiener criterion, 
or, in a still more general nonlinear setting, the axiomatic Perron theory of Heinonen–Kilpel\"ainen–Martio \cite{LSW,Frostman,HKM}.  
These approaches, however, rely on the full arsenal of modern potential theory and are considerably heavier than what is needed in the purely harmonic, finite-energy case considered here.

In recent decades there has been renewed interest in comparatively elementary,
Hilbert space based proofs of this equivalence.
Namely, for continuous boundary data admitting a continuous $H^{1}$-extension,
the results of Simader, Hildebrandt and
Arendt–Daners all identify the variational solution
with the classical Perron solution \cite{Simader,Hildebrandt,ArendtDaners}.  
Simader developed a general Sobolev framework for Perron's construction.
Hildebrandt subsequently gave a shorter argument in the classical setting by tracking the Dirichlet energy directly within Poincar\'e's balayage method
and by invoking weak convergence in $H^1$.  
Arendt and Daners then treated a strictly larger class of boundary data by assuming an extension
$\Phi\in C(\bar\Omega)$ with $\Delta\Phi\in H^{-1}(\Omega)$.
Restricted to continuous data
with a continuous finite energy extension, all three approaches ultimately prove
the same equivalence that is our concern here.

The contribution of the present note is that, in this finite energy setting, the
identification can be obtained by a much shorter argument which works
directly with the Dirichlet minimizer and Wiener's exhaustion, and uses only
strict convexity of the energy, Friedrichs' inequality, and Weyl's lemma.  
In particular, we do not need to adapt Poincar\'e's balayage or Perron's envelope to the Sobolev setting and we do not need weak convergence.  Nor do we need to introduce Poisson problems with right–hand sides in $H^{-1}(\Omega)$.
Wiener's exhaustion is used only in its classical form, and all Sobolev input
is contained in a very simple energy argument.
The resulting
proof could already have been available in the early Sobolev era.

The rest of the paper is organised as follows.  In Section~2 we state the main
theorem and recall the precise definition of Perron's solution.  Section~3
contains the proof, which combines a simple energy argument on an increasing
family of subspaces with Wiener's exhaustion by regular subdomains.  
We then end with a short concluding section.

\section{Main theorem}

Our main theorem is as follows.

\begin{theorem}\label{thm:main}
Let $\Omega \subset \mathbb{R}^n$ be a bounded domain and
let $g \in C(\partial\Omega)$ admit a continuous finite-energy extension
$\phi \in H^{1}(\Omega)\cap C(\bar\Omega)$.
Let $u \in H^{1}(\Omega)$ be the unique minimiser of
\[
  E(v) = \int_{\Omega} |\nabla v|^{2}\,dx,
  \qquad v - \phi \in H^{1}_{0}(\Omega).
\]
Let $h_g$ be the Perron solution with boundary data $g$. 
Then $u = h_g$ almost everywhere in $\Omega$.
\end{theorem}

\begin{remark}[Perron solution]
For completeness we recall the precise definition of Perron's
solution.  The upper Perron class for $g$ consists of all subharmonic functions
$v$ in $\Omega$ such that
\[
  \limsup_{\Omega\ni x\to\xi} v(x) \le g(\xi), \qquad \xi \in \partial\Omega.
\]
The Perron envelope is the pointwise supremum
\[
  h_g(x) := \sup\{v(x) : v \text{ in the upper Perron class}\}, \qquad x \in \Omega.
\]
Since $g$ is continuous on the bounded set $\partial\Omega$, the envelope $h_g$
is finite everywhere and harmonic in $\Omega$; $h_g$ is the Perron solution
with boundary data $g$.
\end{remark}

\begin{remark}[Interior regularity]
By Weyl's lemma, the weakly harmonic function $u$ coincides almost everywhere
with a smooth harmonic function in $\Omega$.  Replacing $u$ by this smooth
representative, we see that $u$ and $h_g$ are classical harmonic functions
which agree almost everywhere in $\Omega$, hence $u = h_g$ pointwise in $\Omega$.
\end{remark}

\begin{remark}[Boundary regularity]
Once the interior equality $u=h_g$ is known, the boundary behaviour of $u$
is determined by the classical potential theory for Perron solutions.
In particular, if $\xi\in\partial\Omega$ is Wiener regular, then
\[
  \lim_{\Omega\ni x\to\xi} u(x) = g(\xi).
\]
\end{remark}

\begin{remark}[Finite energy]
The assumption that the boundary data $g$ admits an extension $\phi\in H^{1}(\Omega)\cap C(\bar\Omega)$ encodes that the corresponding harmonic solution has finite Dirichlet integral.
It is well known (Prym-Hadamard phenomenon) that there exist continuous
boundary functions whose harmonic extensions have infinite energy, even on
the unit disc.  For such data a straightforward Dirichlet principle on
$H^{1}(\Omega)$ cannot even be formulated.
\end{remark}

\begin{remark}[Trace spaces]
On a regular domain (for example a bounded Lipschitz domain) the trace
operator maps $H^1(\Omega)$ continuously onto $H^{1/2}(\partial\Omega)$.
Thus a boundary function $g$ has some $H^1$-extension if and only if
$g\in H^{1/2}(\partial\Omega)$.  Our hypothesis is stronger: we require the
existence of a \emph{single} extension $\phi$ which is both in $H^1(\Omega)$
and continuous on $\bar\Omega$.  In general, it is not enough that $g$
separately admits an $H^1$-extension and a continuous extension; we need one
and the same function $\phi$ to have both properties.
\end{remark}

\section{Proof of the main theorem}

We keep the setting of Theorem~\ref{thm:main}.  
We will construct a sequence $\{u_k\}$ of functions in $\Omega$,
such that $u_k\to u$ in $H^1(\Omega)$, and $u_k\to h_g$ locally uniformly in $\Omega$.
This would establish that $u=h_g$ almost everywhere.

We consider a sequence of open sets $\Omega_1\subset\Omega_2\subset\ldots\subset\Omega$,
such that $\Omega=\bigcup_k\Omega_k$ and that each $\Omega_k$ is Wiener regular.
Such a sequence can be easily constructed, for example, by considering rectangular grids in $\mathbb{R}^n$ with mesh-width $2^{-k}$,
and taking $\Omega_k$ to be the domain generated by the mesh-cells that are contained in $\Omega$.
It is immediate by the exterior cone condition that those $\Omega_k$ are Wiener regular.
The dependence on the exterior cone condition can also be removed, if one prefers, by smoothing out the reentrant corners and edges of $\Omega_k$.
Then for each $k$, we define the function $u_k\in C(\Omega)$ to be the solution of 
\begin{equation}
\begin{cases}
\Delta u_k=0&\text{in}\ \Omega_k, \\
u_k=\phi&\text{on}\ \Omega\setminus\Omega_k.
\end{cases}
\end{equation}
By Wiener’s exhaustion theorem, $u_{k}$ converges locally uniformly in $\Omega$ to the Perron solution $h_{g}$.
Thus it only remains to show that $u_k$ converges to $u$ in $H^1(\Omega)$.

To this end, first observe that $u_k$ minimizes the energy integral $E$ over the set
\begin{equation}
\mathcal{A}_{k} =\{v\in H^1(\Omega): v-\phi\in H^1_0(\Omega_k)\} .
\end{equation}
Here we identify $H^1_0(\Omega_k)$ as a subspace of $H^1(\Omega)$, consisting of functions that vanish outside $\Omega_k$.
On the other hand, $u$ minimizes the energy integral $E$ over the set
\begin{equation}
\mathcal{A} =\{v\in H^1(\Omega): v-\phi\in H^1_0(\Omega)\} .
\end{equation}
It is clear by construction that $\mathcal{A}_k\subset\mathcal{A}_{k+1}\subset\mathcal{A}$ for each $k$.
Thus \(u\) is the global minimizer of \(E\) over the largest admissible class \(\mathcal{A}\),
while each \(u_k\) is the minimizer over the smaller class \(\mathcal{A}_k\) associated with
the inner domain \(\Omega_k\). In order to relate these local minimizers to the
global one, we first note that the sets \(\mathcal{A}_k\) exhaust \(\mathcal{A}\) in the energy
topology.

\begin{lemma}\label{lem:dense}
$\bigcup_{k}\mathcal{A}_{k}$ is dense in $\mathcal{A}$ as a subset of $H^1(\Omega)$.
\end{lemma}

\begin{proof}
Given $v=\phi+w\in\mathcal{A}$ with
$w\in H^{1}_{0}(\Omega)$, approximate $w$ in $H^{1}$ by some
$\varphi\in C_{c}^{\infty}(\Omega)$. 
Then for $k$ large enough, $\operatorname{supp}
\varphi\subset\Omega_{k}$ and then $\phi+\varphi\in\mathcal{A}_{k}$ is close
to $v$ in $H^{1}(\Omega)$.
\end{proof}

This shows that every admissible competitor in \(\mathcal{A}\) can be
approximated in \(H^{1}(\Omega)\) by competitors confined to some \(\Omega_k\).
Since \(u_k\) minimizes the energy over \(\mathcal{A}_k\), this forces the minimal
energies \(E(u_k)\) to converge down to the global minimal value \(E(u)\).
The next lemma makes this precise and upgrades the energy convergence to
strong convergence in \(H^{1}(\Omega)\).

\begin{lemma}\label{lem:energy}
We have $E(u_{k})\to E(u)$ as $k\to\infty$, and
\begin{equation}
\|\nabla(u_{k}-u)\|_{L^{2}(\Omega)}\to0.
\end{equation}
In particular, $u_{k}\to u$ in $H^{1}(\Omega)$.
\end{lemma}

\begin{proof}
Since $\mathcal{A}_{k}\subset\mathcal{A}_{k+1}\subset\mathcal{A}$, we have
$E(u)\le E(u_{k+1})\le E(u_{k})$ and $E(u_{k})$ decreases
to some limit $L\ge E(u)$.  Density of $\bigcup_{k}\mathcal{A}_{k}$ in $\mathcal{A}$
implies
\begin{equation}
E(u)=\inf_{v\in\mathcal{A}}E(v)
   =\inf_{k}\inf_{v\in\mathcal{A}_{k}}E(v)
   =\inf_{k}E(u_{k}),
\end{equation}
and so $L=E(u)$.

For the gradient convergence we use strict convexity.
We have
\begin{equation}
E\Bigl(\frac{u-u_k}{2}\Bigr) =\frac{1}{2}E(u)+\frac{1}{2}E(u_k) - E\Bigl(\frac{u+u_k}{2}\Bigr) .
\end{equation}
Since $\frac12(u+u_{k})\in\mathcal{A}$, minimality of $u$ yields
\begin{equation}
E(u)\le E\Bigl(\frac{u+u_{k}}{2}\Bigr),
\end{equation}
meaning that
\begin{equation}
\|\nabla(u_{k}-u)\|_{L^{2}(\Omega)} = E\Bigl(\frac{u-u_k}{2}\Bigr) \le \frac12E(u_{k}) - \frac12E(u) \to 0 .
\end{equation}
As $u_{k}-u\in H^{1}_{0}(\Omega)$, Friedrichs' inequality yields
$\|u_{k}-u\|_{L^{2}(\Omega)}\to0$, so $u_{k}\to u$ in
$H^{1}(\Omega)$.
\end{proof}

Since $u_{k}\to u$ in $H^{1}(\Omega)$, along a subsequence we have $u_k(x)\to u(x)$ for almost every $x\in\Omega$.
On the other hand, by construction each $u_k$ is the classical harmonic solution
on $\Omega_k$ with boundary values $\phi$, and by Wiener's exhaustion theorem
$u_{k}\to h_g$ locally uniformly in $\Omega$.  
Thus $u$ and $h_g$ agree almost everywhere, 
and the proof of Theorem~\ref{thm:main} is complete.

\section{Conclusion}

We have given a short Hilbert space proof that, for continuous boundary data
with a continuous finite-energy extension, the Dirichlet minimizer coincides
with the classical Perron solution of the Dirichlet problem.
The argument stays entirely in $H^{1}(\Omega)$, uses only strict convexity
of the Dirichlet energy, Friedrichs' inequality, Weyl's lemma and Wiener's exhaustion by regular
subdomains, and never passes through Poisson problems with distributional
right-hand sides.  In particular, it shows that in the finite-energy class
the Dirichlet principle produces exactly the Wiener-exhaustion (and hence
Perron-Wiener) solution, giving 
a substantially shorter and more elementary proof, 
in the finite-energy class, 
than the approaches of Hildebrandt, Simader, and Arendt–Daners.

We finally note that, in much greater generality, the relationship between
variational solutions and Perron–Wiener–Brelot type solutions has been
developed in the framework of Dirichlet forms and elliptic operators without
the maximum principle; see for instance \cite{Klimsiak,ArendtElstSauter}.
Interestingly, in the classical Laplace case, 
the work \cite{ArendtElstSauter} implies that every continuous boundary function
$g\in C(\partial\Omega)$ admits a continuous extension $\Phi\in C(\bar\Omega)$
with $\Delta\Phi\in H^{-1}(\Omega)$.  Thus, combined with \cite{ArendtDaners},
one obtains that the weak solution considered there coincides
with the Perron solution for arbitrary continuous boundary data.  
The present note is restricted to the Laplacian, continuous boundary data with
a finite-energy extension, and a very elementary $H^{1}$-based argument.


\section*{Acknowledgements}
This work was supported by NSERC Discovery Grants Program.

\end{document}